\documentclass[12pt,a4paper,reqno]{amsart} 
\usepackage{amssymb}
\usepackage{latexsym}
\usepackage{amsmath}
\usepackage{mathrsfs}
\usepackage[X2,T1]{fontenc}
\usepackage{cite}
\usepackage{color}

\usepackage{calc}                   

\def\N{\mathbf{N}}

\newcommand{\Char}{\operatorname{Char}}

\newcommand{\scal}[2]{\langle #1,#2\rangle}
\newcommand{\rr}[1]{\mathbf R^{#1}}

\newcommand{\nn}[1]{\mathbf N^{#1}}
\newcommand{\nm}[2]{\Vert #1\Vert _{#2}}

\newcommand{\sets}[2]{\{ \, #1\, ;\, #2\, \} }

\newcommand{\fy}{\varphi}
\newcommand{\cdo}{\, \cdot \, }

\newcommand{\sing}{\operatorname{sing supp}}

\newcommand{\eabs}[1]{\langle #1\rangle}     

\newcommand{\vrum}{\vspace{0.1cm}}

\newcommand{\WF}{W\! F}

\newcommand{\DiRel}{\thicksim _{{}_{\hspace{-0.3cm} x_0}}}
\newcommand{\cE}{\mathcal E}
\newcommand{\mascS}{\mathscr S}
\newcommand{\maclE}{\mathcal E}
\newcommand{\maclF}{\mathcal F}
\newcommand{\maclS}{\mathcal S}
\newcommand{\mascD}{\mathscr D}

\numberwithin{equation}{section}          

\newtheorem{thm}{Theorem}
\numberwithin{thm}{section}

\newcommand{\rubrik}{}
\newtheorem{prop}[thm]{Proposition}
\newtheorem{cor}[thm]{Corollary}
\newtheorem{lemma}[thm]{Lemma}

\theoremstyle{definition}

\newtheorem{defn}[thm]{Definition}
\newtheorem{example}[thm]{Example}

\theoremstyle{remark}
\newtheorem{rem}[thm]{Remark}              

\author{Stevan Pilipovi\' c}

\address{Department of Mathematics and Informatics,
University of Novi Sad, Novi Sad, Serbia}

\email{stevan.pilipovic@dmi.uns.ac.rs}

\author{Joachim Toft}

\address{Department of Computer science, Mathematics and Physics,
Linn{\ae}us University, V{\"a}xj{\"o}, Sweden
}

\email{joachim.toft@lnu.se}

\title{\textbf {Wave-front sets related to quasi-analytic Gevrey sequences}}



\begin{document}
\begin{abstract}
Quasi-analytic wave-front sets of distributions which correspond to the
Gevrey sequence $p!^s$, $s\in[1/2,1)$ are defined and investigated.
The propagation of singularities are deduced by considering
sequences of Gaussian windowed short-time Fourier transforms
of distributions which are modifications of the original distributions by
suitable restriction-extension techniques. Basic micro-local properties
of the new wave-fronts are thereafter established. 
\end{abstract}

\maketitle
%
%

\section{Introduction}

\par

In the literature it seems to be no (local) wave-front sets which
detect heavier singularities than singularities involved in the analytic wave-front set,
while there are different kinds of wave-front sets detecting milder singularities.
For example, if $\WF _A(f)$, $\WF _t(f)$, $t>1$ and $\WF (f)$ are the wave-front sets
of a suitable (ultra-)distribution $f$ with respect to analyticity, Gevrey class $\mathcal E^t$
and smoothness, respectively, it is well-known that
$$
\WF  (f)\subseteq \WF  _t(f)\subseteq \WF  _A(f).
$$
Here $\mathcal E^t(X)$, $t>1$, is the Roumeu  space of ultra-differentiable
functions which correspond to  the Gevrey sequence $p!^t$. (See also
Section \ref{sec1} for notations.) 
We refer to \cite{bh, CPRT1, GS, Ho1, K1, K2, P, R, sjos} for the spaces of non-quasi-analytic and quasi-analytic ultradifferentiable functions.
Note that $\WF _t(f) $
agrees to wave-front sets  $\WF  _L(f)$ of Section 8.4 in \cite{Ho1}, with
$L_p=p^t$ when $t \ge 1$. In particular, if $t=1$, then $\WF _t(f)=\WF  _A(f)$.

Let us mention that the analysis of various wave-fronts local and global, both defined by H\" ormander, and their applications
for distributions and ultradistributons, has been given in many papers
 \cite{coriascomanicia, coriascojohansontoft, cnr, Ho2, melrose, LR, R,
chungkim, sjos, W}. Note that the 
homogeneous wave-front set, used and studied in \cite{martinez, mns, melrose, mizahura,
nakamura, RZ}, is equivalent to the Gabor wave front as well as to the global one of H\" ormander, recently was 
studied in \cite{rw} and after that by
  \cite{mcrs, coriascosulc, ScW, ScW2}. We also refer to our references
\cite{P, PTT, PTT2, coriascojohansontoft}. 

\par

Actually, we will not compare  wave-fronts or consider some specific application as it is done
in many of cited papers, especially for the Schr\" odinger equations.
In this paper  we define the wave-front set ${\WF }_s(f)$, $s\in[1/2,1)$, for
$f\in \mascD  '(\mathbf R^d).$  This is done by restricting 
$f$ to a ball around  $x_0$ ($f_{res}=f_{|L(B(x_0,r)}$), and then, by
the appropriate estimate of the sequence of short-time Fourier transforms
$(V_{\phi _N} f^{rex})(x_0,\xi )$, $N \in \mathbf N$,
$\xi$ belongs to a cone $\Gamma$. Here $\phi _N =e^{-|\cdo |^2/(4N)}$
and $f^{rex}$ denotes an appropriate extension of $f_{res}$.
Our definition extends the notion of ultra-distribution wave-fronts
for $s=t>1$ and can be accommodated in order to extend the
notion of the analytic wave-front in the case $s=1$ (see Remark \ref{nim}).

\par

We  establish basic properties for the wave-front sets for $s\in[1/2,1)$.
Moreover, we introduce a subspace  $\maclE _{0,\infty}^s(\mathbf R^d)$
of the space of Gevrey ultradifferenciable functions $\maclE ^s(\rr d)$,
and analyze the local regularity of an $f\in
\mascD '(\rr d)$ with respect to both spaces.
We have
\begin{equation}\label{singsuprel}
 \sing_sf\subset \pi_1(\WF _s (f)),
\end{equation}
where $\pi _1$ is the projection $\pi _1(x,\xi ) =x$ from $ \rr {2d}$ to
$\rr d$. Considering the local singularities  with respect to
$\maclE _{0,\infty}^s$, we have
\begin{equation}\label{til}
\pi_1(\WF _s(f))\subset {\sing}_{\infty ,s}f.
\end{equation}
We also show that the wave-front set of $f\in\mascD '(\mathbf R^d)$
decreases with the differentiation as well as 
with the multiplication by a function from $\maclE _{0,\infty}^s(\mathbf R^d),
s\in [1/2,1)$.
For the former property we assume additionally that the Fourier transform of
$f$ is a polynomially bounded  locally integrable function.
Consequently, the wave-front sets here can be applied on
problems involving partial differential equations. 

\par

We prove the basic estimate of the propagation
of the wave-front, $s\in[1/2,1)$ related to   a distribution $f$
and a differential operator with constant coefficients $P(D)$:
$$
\WF _s(P(D)f) \subseteq  \WF _{s} (f)
\subseteq \WF (s,P,f) \cup \mbox{Char}(P),
$$ 
where,  $\WF (s,P,f)$ is a suitable set determined by by the
regularity of $P(D)(f^{rex})$ and the polynomial growth of
the Fourier transform of 
$f^{rex}$.

\par

\section{Gevrey wave-fronts}\label{clasgev}\label{sec1}

\par

In general it is a difficult task to examine wave-front properties of Gevrey regularity
of order $s$, when $s<1$, since the presence of suitable compactly supported
functions of such regularity are absent. In this section we introduce
a new approach in this case, based on a suitable restriction-extension technique
for the involved distributions.

\par

Before the definition of the wave-front sets, we introduce some notations.
In what follows we let $\mathcal F$ be the Fourier transform on $\mascS '(\rr d)$
which takes the form
$$
\widehat f(\xi ) = (\mathcal Ff)(\xi )\equiv (2\pi )^{-d/2}\int _{\rr d}f(x)e^{-i\scal x\xi}\, dx
$$
when $f\in \mascS (\rr d)$. In particular, if
\begin{equation}\label{Eq:Ex0Def}
E_{x_0,N}(x)=e^{-|x-x_0|^2/(4N)},\quad  x\in\mathbf R^d,\ N\in \mathbf Z_+,
\end{equation}
then
\begin{multline}\label{Eq:Ecomp1}
(\mathcal F^{-1}E_{-x_0,N})(\xi)=(\mathcal FE_{x_0,N})(\xi)
\\[1ex]
=(2N)^{d/2}e^{-i\sqrt {2N} \scal {x_0}\xi}e^{-N|\xi|^2},\quad \xi \in \rr d,
\end{multline}
and note that
$$
(2\pi )^{-d/2}(\mathcal FE_{x_0,N})(\xi)=
\frac{N^{d/2}e^{-i\sqrt {2N} \scal {x_0}\xi}e^{-N|\xi |^2}} {\pi^{d/2}}
\to e^{-|x_0|^2/2}\delta (\xi )
$$
as $N\to \infty$ with convergence in $\mascS '(\rr d)$. Here and in what follows,
$\mathbf Z_+$ denotes the positive integers, and $\mathbf N=\mathbf Z_+\cup\{0\}$.
For conveniency we set $E_N=E_{0,N}$.

\par
 
\begin{rem}\label{im}
Recall that the $d$-dimensinal Hermite polynomial of order $\alpha \in \nn d$ is given by
$$
H_\alpha (x) = (-1)^{|\alpha |}e^{|x|^2}\partial ^\alpha (e^{-|x|^2}),\qquad x\in \rr d.
$$
We have
$$
e^{-|x|^2/2}|H_\alpha(x)|\lesssim \left ( \frac 2e \right ) ^{|\alpha| /2}\alpha ^{\alpha /2},
\quad x\in\mathbf R^d, \  \alpha \in \nn d.
$$ 
This implies
\begin{alignat}{2}
e^{|x|^2/8N} |E_N^{(\alpha)}(x)| &\lesssim ({e\sqrt N})^{-|\alpha|}\alpha^{\alpha /2},&\quad 
x&\in \rr d,\ |\alpha| \leq N\in \mathbf Z_+.\label{her}
\intertext{Especially, we have}
|E_N^{(\alpha)}(x)|&\lesssim (e\sqrt N)^{-|\alpha|} \alpha^{\alpha/2}, &
\quad
x&\in \rr d,\ |\alpha|\leq N\in \mathbf Z_+. \label{der}
\end{alignat}
\end{rem}

\par

\begin{rem}\label{rr}
For future references we note that for any $l>0$  there exists $c_l>0$ such that
$$
\nm {\eabs \xi ^l\widehat E_N(\xi)}{L^1}<c_l,\quad N\in \mathbf Z_+.
$$
\end{rem}

\par

\begin{defn}
Let $X,Y\subseteq \rr d$ be open, $f\in \mascD '(X)$ and $g\in \mascD '(Y)$.
Then $g$ is called \emph{$f$-related} at $x_0\in X\bigcap Y$, if $f=g$ in an open neighborhood
of $x_0$. The notation $f \DiRel g$ is used when $g$ is $f$-related at
$x_0$.
\end{defn}

\par

Evidently, $\DiRel$ in the previous definition is an equivalence relation.

\par

\subsection{The definition of the wave-front}

\par

%

We now give the definition of regular points and wave-front sets with respect to
the Gevrey class $s\in [1/2,1)$. Here and in what follows
we let $\eabs \xi =(1+|\xi|^2)^{1/2}$.

\par

\begin{defn}\label{gnd}
Let $s\in[1/2,1)$, $X\subseteq \rr d$ be open, $f\in \mascD '(X)$,
$x_0\in X$ and $\xi_0\in\rr d\setminus\{0\}$. Then $(x_0,\xi_0)$ is
called a \emph{Gevrey regular point of order $s$ for $f$}, if for some
$g\in \mascS '(\rr d)$, an open cone $\Gamma$ of $\xi _0$, $C>0$ and $N_0\in \mathbf Z_+$
such that $f \DiRel g$ and
\begin{equation}\label{w1}
|(\maclF (g\, E_{x_0,N}))(\xi)|
\leq \frac{C^{n+1}n^{sn}}{\eabs \xi ^n}\quad 
\text{when}\quad \xi \in \Gamma,\  n\leq N,
\end{equation}
for every integer $N\ge N_0$.

\par

The complement of the set of Gevrey regular points in
$\mathbf R^d\times(\mathbf R\setminus\{0\})$ is denoted by $\WF _{s}(f)$ and is called the
\emph{$s$-wave-front set of $f$}.
\end{defn}

\par
\begin{rem}\label{multi}
Clearly, with the same asumptions, (\ref{w1}) implies that for every $k\in\mathbf N,$
\begin{equation}\label{1w1}
|(\maclF (g\, E_{x_0,kN}))(\xi)|
\leq \frac{C^{n+1}n^{sn}}{\eabs \xi ^n}\quad 
\text{when}\quad \xi \in \Gamma,\  n\leq N,
\end{equation}
\end{rem}

\par

The following result shows that the condition $n\le N$ in \eqref{w1} can be replaced by
$n\le N+N_1$ for any fixed integer $N_0\ge 0$.

\par

\begin{lemma}\label{LemmaWFChar}
Let $s\in[1/2,1)$, $x_0\in \rr d$, $g\in \mascS '(\rr d)$, $g\in \mascS '(\rr d)$,
$\Gamma \subseteq \rr d\setminus 0$ be an open cone, $N_1\ge 0$ be an
integer and let $N_0 \in \mathbf Z_+$. Then the following conditions are equivalent:
\begin{enumerate}
\item there is a constant $C>0$ such that \eqref{w1} holds for every
integer $N\ge N_0$;

\vrum

\item there is a constant $C>0$ such that
\begin{equation}\tag*{(\ref{w1})$'$}
|(\maclF (g\, E_{x_0,N}))(\xi)|
\leq
\frac{C^{n+1}n^{sn}}{\eabs \xi ^n}\quad 
\text{when}\quad \xi \in \Gamma,\  n\leq N+N_1,
\end{equation}
holds for every integer $N\ge N_0$.
\end{enumerate}
\end{lemma}

\par

\begin{proof}
It is clear that (2) implies (1). In order to prove the reversed
inclusion we only consider the case when
$x_0=0$ and $N_1=1$. The general case follows by similar
arguments and is left for the reader.

\par

We need to prove that if (1) holds, then (2) holds in the case $n=N+1$. If (1) holds, then
\begin{multline*}
|(\maclF (g\, E_{N}))(\xi)| = |(\maclF (g\, E_{N+1}E_{N(N+1)}))(\xi)|
\\[1ex]
\lesssim (N(N+1))^{d/2}\int  |(\maclF (g\, E_{N+1}))(\xi -\eta )|e^{-N(N+1)|\eta |^2}\, d\eta
\\[1ex]
\lesssim C_1^N(N+1)^{s(N+1)}\left (1+\int _{|\eta |\ge 1}\eabs {\xi -\eta}^{-N-1}
e^{-N(N+1)|\eta |^2}\, d\eta \right )
\\[1ex]
\le
\frac {2^{N+1}C_1^N(N+1)^{s(N+1)}}{\eabs \xi ^{N+1}} \left (1+\int _{|\eta |\ge 1}|\eta |^{N+1}
e^{-N(N+1)|\eta |^2}\, d\eta \right )
\\[1ex]
\lesssim
\frac {C_2^NN^{sN}}{\eabs \xi ^{N+1}} \left (1+\int _{|\eta |\ge 1}|\eta |^{N+1}
e^{-N(N+1)|\eta |^2}\, d\eta \right )
\\[1ex]
\asymp
\frac {C_2^NN^{sN}}{\eabs \xi ^{N+1}} \left (1+\Gamma ((N+d+1)/2)(N(N+1))^{-(N+d+1)/2}\right )
\\[1ex]
\lesssim \frac {C_2^NN^{sN}}{\eabs \xi ^{N+1}},
\end{multline*}
for some positive constants $C_1>0$ and $C_2>0$. Hence (2) follows.
\end{proof}

\par

In several results later on, we need that additionally $g$ in Definition \ref{gnd}
could be chosen such that
\begin{equation}\label{fbt}
\nm {\widehat{g}(\xi) \eabs \xi ^{-N}}{L^\infty}<\infty
\quad \text{for some}\ N>0.
\end{equation}

\par


\begin{example}
Let $g(x)=e^{-a|x|^2}, x\in \rr d$. Then 
$$\mathcal F(g(x)E_N(x))(\xi)=(\frac{2N}{4aN+1})^{d/2}e^{-\frac{N}{4Na+1}|\xi|^2}, \xi\in \rr d.
$$
One can simply show that (\ref{gnd}) holds in any cone. We have the similar conclusion for $x_0\neq 0.$
\end{example}

\par

\begin{example}
Let $f_n$ be a sequence of entire functions over $\mathbf C$ and
$\{ s_n \} _{n\in \mathbf Z_+}$ be a strictly decreasing sequence in $[1/2,1)$
tending to $1/2$ as $n\to \infty$.
Let the sequence of restriction of $f_n$ on $\mathbf R$ satisfy $f_n\in
\maclS ^{s_n}_{s_n}(\mathbf R)\setminus \maclS ^{s_{n+1}}_{s_{n+1}}(\rr d),$
where $\maclS ^{s_n}_{s_n}(\mathbf R)$ are Gelfand Shilov spaces.
Denote by $\chi _n$ the characteristic function of the set $(-n,-n+1)
\cup(n-1,n), n\in\mathbf Z_+$. Put $f=\sum_{n=1}^\infty \chi_nf_n$ and
 $g_n=f_n, n\in\mathbf Z_+.$ Then, $f \DiRel g_n$ for every $x_0\in(-n,-n+1)\cup(n-1,n)$. Since $g_n\in\maclS ^{s_n}_{s_n}(\mathbf R\setminus ([-n+1,n-1]\cup I_n))$, where 
$$I_n=\{n,-n,n+1,-n-1,...,n+k,-n-k,...\},$$ we obtain
that
$$
\WF _{s_n}(f)\subset \big (\big[-n+1,n-1]\cup I_n\big)
\times \rr d\setminus \{0\}, n\in\mathbf Z_+.
$$
\end{example}

\par

\begin{example}
Let $f$ be a distribution on $\mathbf R$ such that 
$$
\widehat{f}(\xi)
=
\begin{cases}
e^{-\xi^2/2}, & \xi \geq 0,
\\[1ex]
1, & \xi <0.
\end{cases}
$$
Then, 
$$f(x)=
\sqrt{\frac{\pi}{2}} \delta(x)+
\frac{1}{2i\pi}\mbox{ vp }\frac{1}{x}+(\sqrt{\frac{\pi}{2}} \delta(x)-
\frac{1}{2i\pi}\mbox{ vp }\frac{1}{x})*e^{-x^2/2}
$$
Put $g(x)=f(x), x\in\rr . $ 
Clearly,  
$f \DiRel g$ for every $x_0\in \mathbf R$.
 Moreover, 
 $$\widehat g*
 \widehat{E_{x_0,N}}(\xi)=
 2N^{1/2}e^{-i(2N)^{1/2}x_0\xi}(
 \int_{-\infty}^0e^{-N(\xi-\eta)^2}+
 \int^{\infty}_0e^{-\eta^2/4}e^{-N(\xi-\eta)^2})
 $$
 and, one can see that for every $x_0\in\mathbf R$,  
$(x_0,\xi)\in \WF _s (f)$ when $\xi <0$, while $(x_0,\xi)\notin \WF _s(f)$
when $\xi>0$, for every $s\geq1/2.$

Consider $x_0\neq 0.$ Then we can also take $f \DiRel g_0,$ where
\begin{equation}\label{ge}
g_0(x)= 
\frac{1}{2i\pi}\mbox{ vp }\frac{1}{x}+(\sqrt{\frac{\pi}{2}} \delta(x)-
\frac{1}{2i\pi}\mbox{ vp }\frac{1}{x})*e^{-x^2/2}
\end{equation}
since it is equal to $f$ in every neighbourhood of $x_0$ not containing zero.
The "bad" part is $vp\frac{1}{x}$ has the Fourier transform 
$$\mathcal F({vp\frac{1}{x}})(\xi)=-i\frac{\sqrt{\pi}}{2}\mbox{ sgn }\xi, $$ which, in convolution with $e^{-N\xi^2}$ can not be estimated   as in (\ref{w1}), neither for $\xi<0$ nor for $\xi>0$. The convolution part of $g_0$  in (\ref{ge}) may not compensate the growth of the "bad" part for $\xi<0$ or $\xi>0$, as well.
\end{example}

\par

\begin{rem}\label{nim}
In the case  $t>1, f\in\mascD (\mathbf R^d)$, the product of $f$ 
and any cut-of function $\kappa$ , with a sufficiently small support, belonging
to the space of ultra-differentiable functions  $\mascD^t(\mathbf R^d)$,
equals one in a neighborhood of $x_0$, is a suitable extension leading to
the same definition of $\WF _t(f)$.
\end{rem}

\par

\begin{rem}
In the case $s=1$, for  the analytic wave-front one has to use  
a suitable sequence of $g_N\in \mascS '(\rr d), n\in\mathbf N_+,$  such that \eqref{w1} should  be changed into
$$
| \maclF (g_N)(\xi)|=|\langle g_N,
e^{-i\scal \cdo \xi}\rangle| \leq \frac{C^{n+1}n^{sn}}{\eabs \xi ^n},
$$
$$
\xi \in \Gamma,  n\leq N, N_0<N\in\mathbf Z_+,
$$
where $g_N=f\kappa_N$, and $\kappa_N$ is a
sequence of compactly supported smooth functions equals one in a neighborhood
of $x_0$ such that for some $C>0$,
$$
|{\kappa_N^{(\alpha)}}(x)|\leq  (CN)^{|\alpha|},\quad
x\in\mathbf R^d,\ |\alpha|\leq N,
$$
see (8.4.5) in \cite[Section 8.4]{Ho1}.
\end{rem}

\par

\begin{rem}\label{cl}
Let $(x_0,\xi_0)\notin \WF _s(f)$. If $y\in B(x_0,r)$ and $\eta
\in \Gamma$, then $(y,\eta)\notin \WF _s(f)$. Thus, $\WF _{s}(f)$
is a closed set of $\mathbf R^d\times (\mathbf R^d\setminus \{ 0 \} )$.
\end{rem}

\par

\subsection{Basic properties}

The next result links the $s$-wave-front sets to Gevrey regularity of order $s\geq 1/2$.

\par

\begin{prop}\label{uljedan smer}
Let $s\in [1/2,1)$, $X\subseteq \rr d$ be open, $f \in \mascD '(X)$ and $x_0\in X$. 
Assume that there are $k\in \mathbf Z_+$, $N_0 \in \mathbf Z_+$, $C>0$ and $g\in \mascS '(\rr d)$
such that $f \DiRel g$ and
\begin{equation}\label{uNDprimebound}
| (\maclF(g E_{x_0,N} )) (\xi)| \leq \frac{C^{n+1}n^{sn}}{\eabs \xi ^n},
\quad \xi \in \rr d, \  n \leq N,
\end{equation}
for every $N\ge N_0$. Then
\begin{equation}
\label{uinEs}
\sup_{x\in U}
| D^{(\alpha) } f(x)|\le C^{|\alpha | +1} \alpha!^s,\quad
 \alpha \in  \N ^d,
\end{equation}
for some open neighborhood $U$ of $x_0$.
\end{prop}

\par

\begin{proof}
We only prove the result in the case when $N_0=1$, $x_0=0$. The general
case follows by similar arguments and is left for the reader.

\par

We have $f=g$ on $U=B_r(0)$ for some choice of $r>0$. Let
$\alpha \in \nn d$, $x\in U$ and let $C_1>C$. Then
\begin{equation}\label{Eq:FactorFrac}
\sup _{\alpha \in \nn d}\left (\frac{C^{|\alpha|+d+1}(|\alpha|+d+1)!^s}
{C_1^{|\alpha|}|\alpha|!^s} \right )<\infty ,
\end{equation}
and
\begin{equation*}
|D^\alpha (f(x)E_N(x))|
=
|D^\alpha (g(x)E_N(x))|
\lesssim I_1+I_2,
\end{equation*}
where
\begin{align*}
I_1 = \left | \int _{|\xi |\le 1}\xi^\alpha
(\maclF ( g E_N)) (\xi) e^{i \scal x\xi }\, d\xi \right |
\intertext{and}
I_2 = \left | \int _{|\xi |\ge 1}\xi^\alpha
(\maclF ( g E_N)) (\xi) e^{i \scal x\xi }\, d\xi \right | .
\end{align*}

\par

By \eqref{uNDprimebound} we get
$$
I_1\le C,\quad \alpha \in \nn d.
$$

\par

In order to estimate $I_2$ we let $n=|\alpha |+d+1$ and let $N>n$. Then
\eqref{Eq:FactorFrac} gives
$$
I_2\le C^{n+1} n!^s \int _{|\xi |\ge 1}
|\xi |^{|\alpha |-n}\, d\xi ,
$$
which implies that 
$$
\nm {D^\alpha (g E_{N})}{L^\infty (U)} \leq
C_1^{|\alpha|+d+1}(|\alpha |+d+1)!^s
\le
{C_2^{|\alpha|+1}|\alpha|^{s|\alpha|}},\quad
|\alpha | \leq N,
$$
for some positive constants $C_1$ and $C_2$. Letting $N\to \infty$, the left-hand side converges to
$\nm {D^\alpha g}{L^\infty (U)} =\nm {D^\alpha f}{L^\infty (U)}$,
and \eqref{uinEs} follows.
\end{proof}

\par

We also consider spaces as in the following definition.

\par

\begin{defn}\label{Def:SpecialSpaces}
Let $s\ge \frac 12$. Then,
\begin{enumerate}
\item $\maclE^s_{0,\infty}(\rr d)$ consists of all $\fy \in C^\infty (\rr d)$ such that
$\widehat \fy \in L^\infty (\rr d),\alpha \in \nn d$, and
\begin{equation}\label{1tri}
\nm \fy {\maclE^s_{0,\infty ,h}}\equiv  \sup _{\alpha \in \nn d}
\frac{h^{|\alpha |}\nm {\eabs \xi ^{|\alpha |}
\widehat {\fy}(\xi)}{L^\infty(\mathbf R^d)}}{\alpha!^s}<\infty ,
\end{equation}
for some $h>0$;

\vrum

\item $\maclE^s_\infty (\mathbf R^d)$ consists of all
$\fy \in C^\infty (\rr d)$ such that
\begin{equation}\label{00tri}
\nm \fy{\maclE^s_{\infty ,h}}\equiv \sup_{\alpha \in \nn d}
\frac{h^{|\alpha|}\nm {{\fy}^{(\alpha)}}{L^\infty(\rr d)}}{\alpha ! ^s}<\infty.
\end{equation}
for some $h>0$;
\end{enumerate}
\end{defn}

\par

If $A$ and $B$ are topological spaces, then $A\rightarrow B$ means
that $A\subseteq B$ and that the injection map from $A$ to $B$ is continuous, while $A\hookrightarrow B$
 additionally means that $A$ is dense in $B$.

\par

\begin{prop}\label{tri}
Let $s\ge \frac 12$. Then $\maclE^s_{0,\infty}(\rr d)\hookrightarrow
\maclE^s_{\infty}(\rr d) \rightarrow \maclE^s(\rr d)$.
\end{prop}

\par

\begin{proof}
The second embedding is an immediate consequence of the definition.
Let $\fy \in \maclE ^s_{0,\infty}(\rr d)$.  By \eqref{1tri} and the fact that $|\alpha |!
\le d^{|\alpha |}\alpha !$ we have, with suitable $h_1>0$,
$$
\frac{h^\alpha |\fy^{(\alpha)}(x)|}{\alpha!^s}
\leq \int 
\frac{h_1^{|\alpha|+d+1}|\widehat{\fy}(\xi)|\eabs \xi ^{|\alpha|+d+1}}{|\alpha|!^{s+d+1}}
\frac{d\xi}{\eabs \xi ^{d+1}}<\infty.\qedhere
$$

\par

Let $\delta_N=(\pi^{-1}N)^{d/2}e^{-N|\xi|^2}, N\in\mathbf N$ and $\theta\in\mathcal E^s_\infty (\rr d).$
Then $\theta_N=\delta_N*\theta$ is a sequence in $\maclE^s_{0,\infty}(\rr d)$ which converges to $\theta$ in
$\maclE^s_{\infty}(\rr d)$ as $N\rightarrow \infty.$
For the proof we have to use the fact that $||\delta_N||_{L^1(\rr d)}=1, N\in\mathbf N$ and
$$||(\delta_N*\theta)^{(\alpha)}||_{L^1(\rr d)}\leq ||\delta_N||_{L^1(\rr d)}||\theta^{(\alpha)}||_{L^\infty(\rr d)}
$$
and
$$||\eabs \xi ^{|\alpha|}\widehat{\delta_N*\theta}||_{L^\infty(\rr d)}=||\mathcal F^{-1}(\delta_N*\theta)^{(\alpha)}||_{L^\infty(\rr d)}\leq
c ||(\delta_N*\theta)^{(\alpha)}||_{L^1(\rr d)}.
$$
\end{proof}

\par

We have now the following wave-front result.

\par

\begin{prop}\label{prop}
Let $f \in \mascD '(\rr d)$, $P$ be a polynomial on $\rr d$,
and let $\fy \in \maclE _{0,\infty}^s(\rr d)$.
Then the following is true:
\begin{enumerate}
\item if $(x_0,\xi _0)\notin \WF _s (f)$ and $f \DiRel g$ for some $g\in \mascS '(\rr d)$
such that \eqref{fbt} holds, then $(x_0,\xi _0)\notin \WF _s(\fy  f)$.

\vrum

\item $\WF _s(P(D) f)\subseteq \WF _s (f)$.
\end{enumerate}
\end{prop}

\par

\begin{proof}
Assume that $f$ is Gevrey $s$-regular
at $(x_0,\xi _0) \in \rr d \times (\rr d \setminus\{0\})$, and choose
$g \in \mascS '(\rr d)$ such that $f \DiRel g$ and \eqref{fbt} hold. We shall
prove that $(x_0,\xi _0)\notin \WF _s(\fy  f)$ and $(x_0,\xi _0)\notin
\WF _s(P(D) f)$.
We only prove these relations in the case $x_0=0$ and $k=N_0=1$ in Definition
\ref{gnd}. The general case follows by similar
arguments and is left for the reader.

\par

(1) We have $(0,\xi _0)\notin \WF _s (f)$.
We shall apply the standard technique as in \cite[Lemma 8.1.1]{Ho1}. 
Let $\Gamma$ be an open cone such that $\xi _0\in \Gamma$ and that \eqref{w1}
holds, and let $\Gamma _1\subseteq \Gamma \cup \{ 0\}$
be a closed cone with $\xi _0$ as an interior point. Then with a suitable $c\in (0,1)$,
\begin{equation}\label{gam}
\begin{gathered}
\xi\in \Gamma_1, |\xi|>1 \mbox{ and } |\xi-\eta |\leq c|\xi |
\quad \Rightarrow \quad \eta \in \Gamma,
\\[1ex]
|\xi-\eta|\leq c|\xi | \quad \Rightarrow \quad |\xi|\leq (1-c)^{-1}|\eta|.
\end{gathered}
\end{equation}

\par

We have
$$
(\maclF (\fy  gE_N))(\xi)= I_1(\xi)+I_2(\xi),
$$
where
\begin{align*}
I_1(\xi ) &= \int_{|\xi -\eta|\leq c|\xi|} \widehat{\fy }(\xi -\eta ) (\maclF (gE_{N}))(\eta )
\, d\eta, 
I_2(\xi ) &= \int_{|\eta|\geq c|\xi|}
\widehat{\fy }(\eta) (\maclF (gE_{N}))(\xi-\eta)\, d\eta 
\end{align*}
and $\xi \in \Gamma _1\subset \Gamma$. We need to estimate $|I_1(\xi )|$
and $I_2(\xi )$ and start with the former one.

\par

We have
\begin{multline}\label{dod1}
\sup_{\xi\in\Gamma_1}|\xi|^nI_1(\xi)\leq \sup_{\xi\in\Gamma_1}|\xi |^n\sup
_{|\xi -\eta |\leq c|\xi|}| (\maclF (gE_{N}))(\eta)|
\int |\widehat{\fy }(\xi-\eta)|\, d\eta
\\[1ex]
\leq C_1(1-c)^{-n}\sup_{\eta\in\Gamma}|\eta |^n| (\maclF (gE_{N}))(\eta)|
\leq C^{n+1}n^{sn}, \quad n\leq N.
\end{multline}
Here the second inequality follows from the fact that $|\xi |\le (1-c)^{-1}|\eta |$ when
$|\xi -\eta |\le c|\xi |$.

\par

Next, we estimate  $|I_2 (\xi )|$. By \eqref{fbt} we get
\begin{multline}\label{fr}
\nm {\eabs \xi ^{-l}\mathcal F(gE_{N})(\xi)}{L^\infty}
\leq
\nm {(\eabs \xi ^{-l} (\maclF  g)(\xi))*(\eabs \xi ^l\widehat E_N(\xi))}{L^\infty}
\\[1ex]
\leq \nm {\eabs \xi ^{-l} (\maclF g)(\xi)}{L^\infty}\nm {\eabs
\xi ^l\widehat E_N(\xi)}{L^1}<\infty .
\end{multline}

Let $n\leq N$. It follows from \eqref{uNDprimebound}, \eqref{fr}
and the assumptions on $\fy$ that if $C>0$ is chosen large enough, then
\begin{multline*}
\left | \frac{|\xi|^n}{C^{n+1}n!^s}I_2(\xi) \right |
\leq \int_{|\eta|\geq c|\xi|}\frac {|\eta ^n
\widehat {\fy }(\eta)|}{C^{n+1}n!^s}
\eabs {\xi-\eta} ^l\eabs {\xi-\eta} ^{-l}| (\maclF (gE_N))(\xi-\eta)|\, d\eta
\\[1ex]
\leq C_1\int_{|\eta|\geq c|\xi|}\frac{|\eta ^n \widehat {\fy }(\eta)|}{C^{n+1}n!^s}
\eabs \eta ^{l+d+1}\frac{d\eta}{\eabs \eta ^{d+1}}
\\[1ex]
\leq C_2\sup _{|\eta|>c|\xi|}\left ( \frac{|\eta ^{n+r}\widehat {\fy }(\eta)|}
{C^{n+1}n!^s}\right )
<\infty ,
\end{multline*}
where $r>l+d+1$, for some constants $C_1$ and $C_2$. This gives
\begin{equation}\label{dod2}
|I_2(\xi)|\leq \frac {C^{n+1}n^{sn}}{|\xi|^n},\quad  \xi \in \Gamma_1,\ n\leq N
\end{equation}
for some constant $C>0$. The assertion now follows by combining
\eqref{dod1} and \eqref{dod2}.

\par

(2) The assertion follows if we prove $(0,\xi)
\notin \WF _s(\partial _{x_k}f)$, $1\le k\le d$.
Let $\xi\in\mathbf R\setminus\{0\}.$ We have
\begin{equation}\label{dod3}
\mathcal F((\partial_{x_k}g)E_N)(\xi)=i\xi_k\mathcal F(gE_N)(\xi)-
\frac{1}{2N}\mathcal F(x_kgE_N)(\xi ). 
\end{equation}
We estimate the terms on the right-hand side separately.

\par

In view of Lemma \ref{LemmaWFChar}, the first term in the right-hand side of
\eqref{dod3} can be estimated as
$$
| i\xi_k\mathcal F(gE_N)(\xi) | \le \frac{C^{n+1}n^{sn}}{\eabs \xi ^{n-1}}
\lesssim
\frac{C_1^{n}(n-1)^{s(n-1)}}{\eabs \xi ^{n-1}},
\quad \xi \in \Gamma ,\ n\le N+1,
$$
for some constants $C$ and $C_1$. Hence
\begin{equation}\label{Eq:FirstTerm1}
| i\xi_k\mathcal F(gE_N)(\xi) | \le \frac{C^{n+1}n^{sn}}{\eabs \xi ^n},
\quad \xi \in \Gamma ,\ n\le N,
\end{equation}
for some constant $C$.

\par

Differentiating \eqref{Eq:Ecomp1}, using that 
$$
|\mathcal F(x_kE_N)(\xi)|=|\partial_{\xi_k}\mathcal F(E_N)(\xi)|
,\quad  \xi \in \rr d,
$$
and taking $\sqrt N\xi$ as new variables of integration we obtain
$$
\frac 1{2N}\int | \mathcal F(x_kE_N)(\xi)|\, d\xi =C_1\int \xi_ke^{-N|\xi|^2}N^{d/2}\, d\xi 
\le C_2N^{-1/2}\le C_2,
$$
for some constants $C_1$ and $C_2$. Hence, substituting $N$ by $2N$
we get
$$
\frac{1}{2N}\int | \maclF (x_kE_{2N})(\xi)|\, d\xi <C,
$$
where $C$ is independent of $N$.
Thus, if $\Gamma _1$ and $\Gamma$ are the same as in the first
part of the proof, it follows from that part that for the second therm in \eqref{dod3} we have, using in the end (\ref{1w1}),
\begin{multline*}
\sup _{\xi \in \Gamma _1} \left (
\frac{\eabs \xi ^n}{C^{n+1}n!^s}|\mathcal F(g(\partial_{x_k}E_N))(\xi)| \right )
\\[1ex]
\frac 1{2N}\sup _{\xi \in \Gamma _1} \left (
\frac{\eabs \xi ^n}{C^{n+1}n!^s}||\mathcal F(gE_{2N})|*|\mathcal F(\partial_{x_k}E_{2N})|(\xi)|
\right )
\\[1ex]
\leq
\sup _{\xi \in \Gamma} \left ( \frac{\eabs \xi ^n}{C^{n+1}n!^s}|\maclF (gE_{2N})(\xi)| \right )
\frac{1}{2N}\int| \maclF (x_kE_{2N}))(\xi)|\, d\xi <C,\quad n\leq N.
\end{multline*}
where $C>0$ is a suitable constant not depending on $n$ and $N$,
and the assertion follows.
\end{proof}

\par

\section{Local regularity}\label{sec2}

\

\begin{prop}\label{drugi smer}
Let  $U\subseteq \rr d$ be open, $x_0\in U$,  
$f\in \in \mascD '(U)$, and assume that $g\in \maclE _{0,\infty}^s (\rr d)$
be such that $f \DiRel g$. 
Then there exists $C >0$ such that 
\eqref{uNDprimebound} holds for $\maclF (gE_N)$.
\end{prop}

\par
 
\begin{proof}
Let $n\leq N\in\mathbf R^d$. Then
$$
\sup_{\xi\in\mathbf R^d}|\eabs \xi ^n
(\maclF (gE_N))(\xi)|\le \sup_{\xi\in\mathbf R^d}|\eabs \xi ^n \widehat{g}(\xi )|
\nm {\eabs \xi ^n \widehat E_N (\xi)}{L^1(\mathbf R^d)},
$$
and the result follows from the fact that $\nm {\eabs \xi ^n\widehat E_N(\xi)}
{L^1(\mathbf R^d)}<c$, for some $c$ which is independent of $n$ and $N$.
\end{proof}

\par

%
As a consequence we have the following. Here ${\sing} _{\infty, s} f$ is
the set of points $x\in\mathbf R^d$ such that
it does not exist any $g\in \maclE _{0,\infty}^s(\rr d)$ such that $f \DiRel g$.

\par

\begin{cor}\label{contra}
Let U be open and $x_0\in U$, $f\in \mascD '(U)$ and $g\in \maclE _{0,\infty}
^s(\mathbf R^d)$ be such that $f \DiRel g$. Then $(x_0,\xi)\notin \WF  _s(f)$,
for any $\xi\in\mathbf R^d\setminus\{0\}$.
In particular, \eqref{til} holds.
\end{cor}

\par

Definition \ref{gnd} and the compactness of the sphere $\mathbf S^{d-1}$ imply the next proposition.

\par

\begin{prop}\label{cont}
Let $f\in\mascD '(\rr d)$ and  $x_0\in \mathbf R^d$ be such that  $(x_0,\xi)
\notin{\WF _s} (f)$ for every $\xi\in\mathbf R^d\setminus\{0\}$.
Then there exists an open neighbourhood $U$ of $x_0$ such that  $f\big \vert _{U}
\in \maclE ^s(U)$.
\end{prop}

\par

Now, we compare the projections of these sets with the singular support
with respect to $\cE ^s$.

\par

\begin{thm}\label{lemica}
Let $s\in [1/2,1)$,  $f \in \mascD ' (\mathbf R^d)$,
$K\subseteq \rr d$ be compact, and let $F$ be a closed cone.
If $ {\WF }_s (f) \cap (K \times F) =
\emptyset$. Then there exist an open set $U$, an open cone
$\Gamma$ and $g\in \mascS '$ such that $f=g$ on $U$, 
$
K\times F\subset U\times\Gamma
$
and for some $C>0$,
\begin{equation}\label{2hatchiu}
| (\maclF (gE_{N})) (\xi)| \leq C^{n+1} \frac{n!^s}{\eabs \xi ^{n}}, \xi \in \Gamma, 
\quad n\leq  N\in \N.
\end{equation}
\end{thm}

\par

\begin{proof}
Let $K=\{x_0\}, \xi_0\in F=\Gamma _{\xi_0}$ be a closed conic neighbourhood of
$\xi _0$ contained in an open cone $\Gamma_0$ such that \eqref{w1} holds in $\Gamma_0.$

\par

For the sake of simplicity, assume $x_0=0$ and $N_0=1$.
We  conclude the assertion in  the case 
$K=\{x_0\}, F=\Gamma_{\xi_0}, U=B(x_0,r)$ and $\Gamma=\Gamma_0.$ 

\par

Let  $K=\{x_0\}$ and  $F$ be a closed cone. We note that the intersection of $F$ with
the unit sphere is compact. Hence we may choose a finite number
of balls, $B(x_0,r_{x_0,\xi _j})$, closed cones $\Gamma_{\xi_j}$ compactly included in open cones 
$\Gamma_j, j=1,\dots ,k$ (\eqref{w1} holds in $\Gamma_j$),  then
take for $U$ the intersection of open balls, and for $\Gamma $,
$
\Gamma \equiv \bigcup _{j=1}^k\Gamma _j.
$
 
\par

Finally, since $K$ is compact, we may cover $K$  by finite
number of open balls  $B_{x_k}$, $ k = 1,\dots, m$,
and repeat the procedure for every ball. 
\end{proof}

\par

The following result links the  ${\sing}_s$  with th $s$-wave-front set.

\par

\begin{thm} \label{th16}
Let $s\in [1/2,1)$ and $f\in \mascD  ' (\mathbf R^d)$. Then \eqref{singsuprel} holds.
\end{thm}

\par

\begin{proof}
%
Assume that $(x_0, \xi_0)\notin \WF _s (f)$ for all
$ \xi_0 \in \rr d \setminus 0 $. Then there is a neighborhood
$U$ of $ x_0 $ such that $ \WF _s (f) \cap
(U \times \rr d) = \emptyset$ and $g\in \mascS '$ equal to $f$ on $U$.
As before, we only consider the case when
$x_0=0$, $k=1$ and $N_0=1$. Then
\begin{equation}\label{1hatuN}
|(\maclF (g E_N)) (\xi) |\leq \frac{C^{n+1} n!^s}{\eabs \xi ^{n}},\ 
\xi \in \rr d,\  n\leq N,\quad N \in \mathbf Z_+,
\end{equation}
holds for some $C>0$. By Proposition \ref{uljedan smer},
we conclude that $g\in  {\cE} ^s(\mathbf R^d)$.
That is,  $ 0 \not\in \sing _s f$. 
\end{proof}

\par

The next  statement
is a straight-forward consequence of the definition and previous
results. 

\par

\begin{prop}\label{th17}
Let $s\in (0,1]$, $f \in \mascD ' (\mathbf R^d)$ and
$ 1/2\leq s_1 < s_2 \leq 1$. Then
${ \WF _{s_2} (f)} \subset \WF _{s_1} (f).$
\end{prop}

\par

\section{Wave-front of $P(D)f=h$}\label{sec3}

\par

Let $D^\alpha=(-i)^{|\alpha|}\partial^{\alpha_1+...+\alpha_d}/(\partial_{x_1}^{\alpha_1}...\partial_{x_d}^{\alpha_d})$,
 $P(D)=\sum_{|\alpha|\leq m}a_\alpha D^{\alpha}$ be a differential operator with constant coefficients, $P_m(\xi)=\sum_{|\alpha|= m}a_\alpha \xi^{\alpha}$ its principal symbol, and
$f\in \mascD '(\mathbf R^d)$. Recall, $\Char (P)$ is defined by
$\mbox{Char} (P)=\{(\xi)\in T^*(X)\setminus 0, P_m(\xi)\neq 0\}$.

\par

\begin{defn}\label{P-u}
The set $\mbox{Reg}(s,P,f)$ consists of all points $(x_0,\xi_0)\in
\rr d \times (\rr d\setminus\{0\} )$ such that for some $v\in \mascS '(\rr d)$,
the following conditions hold true: 
\begin{enumerate}
\item $f \DiRel g$ and \eqref{fbt} holds true;
%
%

\vrum

\item for some open conical neighborhood $\Gamma$ of $\xi _0$,
some $N_0\in \mathbf Z_+$ and $C>0$, \eqref{w1} holds
with $P(D)g$ in place of $g$, for every $N\ge N_0$. 
%
%
\end{enumerate}

\par

The  complement of $\mbox{Reg}(s,P,f)$ is denoted by $\WF (s,P,f)$.
\end{defn}
 
\par
 
Evidently, $\WF (s,P,f)$  is a closed set. 

\par

\begin{rem}\label{hyp} 
The assumption 
\eqref{hyp} is needed in the proof of the next theorem and
it is an open problem whether this theorem holds in a more
general case.
\end{rem}

\par

\begin{thm}\label{glav}
Let $P(D)$ be a differential operator with constant coefficients and
$f\in \mascD '(\mathbf R^d)$. Then, for $s\in[1/2,1)$,
\begin{equation}\label{pwf}
\WF _{s}(P(D)f)\subset \WF _{s}(f) \subset 
\WF (s,P,f)\cup \mbox{Char}(P).
\end{equation}
\end{thm}

\par
\begin{rem}\label{small} Let $A>$ and $P_A(D)=\frac{1}{A}P(D)$. We can simply conclude that 
(\ref{pwf}) holds for $P(D)$ if and only if it holds for $P_A(D).$ This remark will be important in the proof which is to follow
when we need to have that 
$r_0=\sum_{|\alpha|\leq m}|a_\alpha|$ is enough small. This will be explained in the proof. 
\end{rem}
\begin{proof} 
Assume that $(x_0,\xi_0)$ does not belong to the right-hand side
of \eqref{pwf} i.e. there exist a  neighbourhood $U$  of $x_0$ and
an open conic neighbourhood $\Gamma$ of $\xi_0$
in $\mathbf R^n\setminus\{0\}$ such that
\begin{equation}\label{cond1}
P_m(\xi)\neq 0 \mbox{ in }  \Gamma,\quad (U\times \Gamma)\cap 
{\WF (s,P,f)}=\emptyset .
\end{equation}
We assume that $x_0=0$. We use the notation
$P(D)g=h$ such that $g$ satisfies \eqref{fbt} and consequently
$h$ satisfies \eqref{fbt}, with another exponent and with $h$
in place of $g$. Moreover, $h$ satisfies \eqref{w1}.

\par

We will follow the proof of Theorem 8.6.1 in\cite{Ho1}.
However, we make several important modifications 
which makes this proof different from that of quoted theorem in \cite{Ho1}.

\par

 We consider equation
\begin{equation}\label{e1}
(^tP(D)\fy )(x,\xi)= E_{N}(x)e^{-i\scal x\xi},
\quad x, \xi\in\rr d,\ N\in \mathbf Z_+.
\end{equation}
With
$$
\fy (x,\xi)=w(x)e^{-i\scal x\xi}/P_m(\xi),\quad x,\xi\in \rr d,
$$
as in \cite{Ho1}, one pass to an equation of the form 
\begin{equation}\label{e1}
w-Rw= E_{N},\quad R=R_1+\cdots +R_m,
\end{equation}
where $|\xi|^jR_j$ is a differential operator  of order less than or equal to $j$ 
and homogeneous of degree zero with respect to 
$\xi$ when $\xi\in \Gamma, j=1,\dots ,m$. Formally, a solution should
have a form $w=\sum_{j=0}^\infty R^j E_{N}$.

\par

Let $ x, \xi\in\mathbf R^d$ and
\begin{equation}\label{gar}
w_N(x,\xi)=\sum _{p=0}^{2N-m-1}\sum_{j_1+\cdots +j_k=p}
(R_{j_1}\cdots R_{j_k}
E_{N})(x,\xi),\quad  N\in\mathbf Z_+,
\end{equation}
where the composition $R_{j_1}\cdots R_{j_k}$ with $j_1+\cdots +j_k=p$
has the form
\begin{equation}\label{form}R_{j_1}\cdots R_{j_k}=|\xi |^{-p}
\sum_{|\alpha|\leq p}b_\alpha \partial_x^\alpha.
\end{equation}

\par

For the indices $j_1,\dots ,j_k,$ we introduce the set
$$
J_N= \bigcup _{k\ge 1}\sets {(j_1,\dots ,j_k)\in \nn k}
{j_2+\cdots +j_k<N\le j_1+j_2+\cdots +j_k}.
$$
Then
\begin{equation}\label{WRN}
w_N-Rw_N= E_{N}-\sum _{j_1,\dots ,j_k\in J_{2N-m}}
R_{j_1}\cdots R_{j_k} E_{N}.
\end{equation}

\par

By \eqref{WRN} we have
$$
^tP(D)(e^{-i\scal x\xi }w_N(x,\xi)/P_m(\xi))
=e^{-i\scal x\xi }(E_{N}(x)-e_{N}(x,\xi)),
$$
where
\begin{equation}\label{en}
e_{N}(x,\xi)=\sum _{j_1,\dots ,j_k\in J_{2N-m}}
(R_{j_1}\cdots R_{j_k} E_{N})(x,\xi).
\end{equation}
Then
\begin{multline*}
\mathcal F(gE_{N})(\xi)
\\[1ex]
=\mathcal F(g\cdo e_{N}(\cdo ,\xi))(\xi)+
\langle h e^{-i\scal \cdo \xi },w_N(\cdo ,\xi)/P_m(\xi)\rangle,\quad \xi \in \rr d.
\end{multline*}

\par

We need to estimate $e_N$ and begin with estimating $\sigma_p$, the number 
of operators $R_{j_1}\cdots R_{j_k}$, $j_1+\cdots +j_k=p$ of the  form \eqref{form}.
More precisely, we have to find out the number of presentations
$$
p=j_1+\cdots +j_k,\quad  j_i\in\{1,\dots ,m\},\ i=1,\dots ,k
$$
with $k\leq p$. Here $k=p$ when  $j_i=1, i=1,\dots ,p$.  One can find that
(with suitable $c>0$)
\begin{multline}\label{car}
\sigma_p\leq {{2p-1}\choose p}-{{2p-2m-3}\choose {p-m-1}}
\\[1ex]
\asymp
\frac{1}{2}
\left (
\frac{4^p}{\sqrt{\pi p}}-\frac{4^{p-m}}{\sqrt{\pi (p-m)}}
\right )
\leq c4^p.
\end{multline}
Let us explain this rough estimate. The number of $p$ units can be divided into $p$ boxes by  ${2p-1}\choose p$ ways but if one of boxes, at least, has 
$m+1$ units this possibility should be subtracted. One has ${2p-2m-3}\choose {p-m-1}$ such possibilities.

The summation over the set of indices in (\ref{en}),  
can be estimated by 
the number of terms in \eqref{en} multiplied by the maximal one.

\par

Next we estimate the number $s$ of terms in \eqref{en}.
If $p=2N-m-i$, $i=1,\dots ,m-1$, with application of $R_{j_1}$ on
$R_{j_2}\cdots R_{j_k},$ one can rich one of the members of the
sum  in \eqref{en}. The choice of $j_1$ depends on $i$ but the
number of such $j_1$ is less than $m(m-1)/2.$ Thus, by \eqref{car},
and with another constant $c$, we have
\begin{equation}\label{es}
s\leq c4^{2N-m}.
\end{equation}
With the similar argument we estimate $S$,
the number of terms in $w_N$:
\begin{equation}\label{Se}
S\leq c4^{2N-m},
\end{equation}
for some other constant $c$.

\par

With the notation of Remark (\ref{small}), we have
\begin{equation}\label{KSI}
|\xi |^p|R_{j_1}\cdots R_{j_k}E_{N}(x)|
\leq c r_0^p\sup_{|\alpha|\leq p}|\partial^\alpha_xE_{N}(x)|, x\in \rr d.
\end{equation}
Thus, \eqref{es} and \eqref{her} imply
\begin{multline*}
|\xi|^{2N-m}|e_N(x,\xi)|
\\[1ex]
\leq c(4r_0)^{2N-m}
(\frac{1}{e\sqrt {N}})^{2N-m}(2N-m)^{(2N-m)/2}e^{-|x|^2/(8N)}.
\end{multline*}
Now we use Remark(\ref{small}). From the early begining we should assume that $r_0$ is so small so that
 $4r_0/e<1$.Below, we will give one more condition on $r_0$. With this, we have
\begin{multline}\label{h4}
|e_N(x,\xi)|\leq c|\xi|^{-2N+m}\left (\frac{4r_0}{e} \right ) ^{2N-m}e^{-|x|^2/(8N)}
\\[1ex]
\leq c|\xi|^{-2N+m}e^{-|x|^2/(8N)}
\end{multline}

\par

By differentiating  $e_N(x,\xi)$ with respect to $x$ and taking the Fourier
transform with respect to $x$, it follows that if
$s=d+1$ if $d$ is odd or $s=d+2$, if $s$ is even, then
there exists $C>0$ such that
\begin{equation}\label{EN}
\sup _{\eta \in \rr d} \left |
(1+|\eta|^2)^{s/2} (\maclF e _{N,\xi})(\eta )
\right |
%
\leq C\eabs
\xi ^{-2N+m},\quad \xi \in \rr d,
\end{equation}
where $e_{N,\xi}(x)=e_N(x,\xi)$ is considered as a function in
$x$, parameterized by $N$ and $\xi$.

\par

In order to give more details we write
$$
(1-\Delta)^{s/2}=\sum _{|\beta|\leq s}c_\beta\partial_x^\beta ,
$$
and let $K=\sum _{|\beta|\leq s}|c_\beta|$. Then, by \eqref{en}, \eqref{Se}
and \eqref{EN}
\begin{multline*}
|\xi|^p|(1-\Delta)^{s/2}e_N(x,\xi)|
\\[1ex]
\le
\sum_{|\beta|\leq s}|c_\beta|\sum_{j_1,j_2,\cdots ,j_k\in J_{2N-m}}
|(R_{j_1}\cdots R_{j_k} \partial_x^\beta E_{N})(x,\xi)|
\\[1ex]
\leq K\sum_{j_1,j_2,\cdots ,j_k\in J_{2N-m}}c r_0^p
\sup_{|\alpha |\leq p,|\beta |\leq s}
|\partial^{\alpha+\beta}_xE_{N}(x)|, x\in \rr d
\\[1ex]
\leq 
cK(4r_0)^{2N-m}
 (\frac{1}{e\sqrt {N}})^{2N-m+s}(2N-m+s)^{(2N-m+s)/2}e^{-|x|^2/(8N)}.
\end{multline*}
Now, by the determined assumption on  $r_0$, we have
$$
\frac{4r_0\sqrt{2N-m-s}}{e\sqrt{N}}\leq 1,
$$
and obtain \eqref{EN}.

\par

By similar arguments it follows that \eqref{EN} holds true also
with the $L^1$ norm on the left hand side, provided the constant
$C$ has been replaced by a larger one if necessary.

\par

%
%
%
%
%
%
%
%
%

Since $g$ satisfies \eqref{fbt}, we have
$$
| \scal {g}{e^{-i\scal \cdo \xi }e_{N}(\cdo ,\xi)} |
= | (\widehat g*\widehat e_{N,\xi} )(\xi) |,\quad \xi \in \rr d,
$$
and
$$
| \scal {g} {e^{-i\scal x\xi }e_{N}} | \leq c\eabs \xi ^{-2N+l+m}.
$$
Thus, for $N_0=l+m$, we have
\begin{equation}\label{e33}
|\langle g,e^{-i\scal \cdo \xi } e_{N,\xi} \rangle |
\leq \frac{C^{n+1}n^{sn}}{\eabs \xi ^n},\quad
\xi \in \rr d,\ n\leq N,\ N>N_0.
\end{equation}
which is the searched estimate.

\par

Similarly concerning estimate of $e_N$, by \eqref{her}, \eqref{gar} and \eqref{Se},
we may conclude that 
$$
\left | \frac{D^\alpha _xw_N(x,\xi)}{P_m(\xi)} \right |
\leq c\eabs \xi ^{-2N}e^{-\frac{|x|^2}{8N}},\quad  x,\xi \in \rr d,
\ |\alpha |\le s
$$
and, for $s=d+1$ or $s=d+2$
there exists  $C>0$ such that
\begin{equation}\label{es1} 
|{P_m(\xi)}|^{-1} \nm {\widehat w_{N,\xi}\cdot \eabs \cdo ^s}{L^1(\mathbf R^d)}
 \leq C\eabs \xi ^{-2N},\quad \xi \in \rr d.
\end{equation}

\par

We shall estimate
$$
\left |\frac{\widehat w_N(x,\xi )}{P_m(\xi )}\maclF (hE_{N})(\xi ) \right |
$$
by using similar arguments as in the proof of Proposition \ref{prop}.
More precisely, let $\Gamma_1\subset\subset \Gamma$. Then (\ref{gam}) holds.
We have
$$
\left |\frac{\widehat w_{N}}{P_m}* (\maclF (hE_{N})) (\xi)\right |
\leq I_1(\xi )+I_2(\xi ),
$$
where
\begin{align*}
I_1(\xi ) &\equiv \int_ {|\eta|\leq c|\xi|}
\left | \frac{\widehat w_{N}(\eta,\xi)}{P_m(\xi )}\right |
| (\maclF ( hE_{N})) (\xi-\eta)|\, d\eta
\intertext{and}
I_2(\xi ) &\equiv \int _{|\eta | > c| \xi |}
\left | \frac{\widehat w_{N}(\eta,\xi)}{P_m(\xi)}\right |
| (\maclF ( hE_{N})) (\xi-\eta)|\, d\eta
\end{align*}

\par

For $I_1$ we have
$$
I_1(\xi )\leq \sup _{|\eta-\xi|<c|\xi|} | (\maclF (hE_{N}))(\eta) |\int
_{|\xi - \eta |\leq c|\xi |} \left |\frac{\widehat w_{N}(\xi -\eta ,\xi )}
{P_m(\xi )}\right |
\, d\eta .
$$
Let $n \leq N$. The estimate \eqref{w1} for $\maclF (hE_{N})(\xi-\eta)$,
\eqref{es1} and \eqref{gam} imply
\begin{multline}\label{1hand}
I_1(\xi )|\xi|^{n}
\\[1ex]
\le
(1-c)^{-d}\sup _{\eta \in \Gamma} | ( \maclF (hE_{N})) (\eta)|
|\eta |^{d}
\int _{|\eta | \geq (1-c)|\xi |}
{{\left | \frac{\widehat w_{N}(\xi -\eta,\xi)}{P_m(\xi)}\right |}}
\, d\eta
\\[1ex]
\leq C^{n+1}n^{sn},\quad n\leq N,\ N>N_0,\ \xi\in \Gamma_1,\ |\xi|>1.
\end{multline}

\par

In order to estimate $I_2$ we use
$$
|(\maclF (hE_{N}))(\xi)|\leq
C\eabs \xi ^{l},\quad \xi \in \rr d.
$$
$$
\int _{|\eta |>c|\xi |}
\left | \frac{\widehat w_{N}(\eta,\xi)}{P_m(\eta,\xi)}\right |
|(\maclF ( h E_{N})) (\xi-\eta)|\, d\eta\leq
$$
$$
\leq \int_{|\eta | >c|\xi |}
\eabs \xi ^l\eabs \eta ^l
\left | \frac{\widehat w_{N}(\eta,\xi)}{P_m(\eta,\xi)}\right |
|\eabs{\xi -\eta} ^{-l}
(\maclF ( hE_{N}))(\xi-\eta)|\, d\eta
$$
$$\leq
C \int_{|\eta |>c|\xi |}((1+c^{-1})^{l}\eabs \eta ^{2l} 
\left | \frac{\widehat w_{N}(\eta,\xi)}{P_m(\eta,\xi)}\right |
\, d\eta,
$$
where we have used the fact that $|\eta |>c|\xi |$ implies
 $|\xi-\eta|\leq (1+c^{-1})|\eta|$.

\par

Let $n<N.$ Then
\begin{multline*}
\eabs \xi ^n\int_{|\eta|>c|\xi|}
\left | \frac{\widehat w_{N}(\eta,\xi)}{P_m(\eta,\xi)}\right |
|(\maclF ( hE_{N}))(\xi-\eta)|\, d\eta
\\[1ex]
\leq C \eabs \xi ^n\int_{\rr d}
\eabs \eta ^{2l} \left | \frac{\widehat w_N(\eta,\xi)}{P_m(\eta,\xi)}
\right |\, d\eta.
\end{multline*}
By similar arguments as in Remark \ref{rr} we get
\begin{multline}\label{dod6}
\eabs \xi ^n\int_{|\eta|>c|\xi|}
\left | \frac{\widehat w_{N}(\eta,\xi)}{P_m(\eta,\xi)}\right |
| (\maclF ( hE_{N}))(\xi-\eta)|\, d\eta.
\\[1ex]
\leq C^{n+1}n^{sn}\eabs \xi ^{n-2N},\quad  \xi\in \rr d,\ n<N,\ N>N_0.
\end{multline}
The result now follows from \eqref{e33}, \eqref{1hand},  and \eqref{dod6}.
\end{proof}

\end{document}